\documentclass{amsart}

\usepackage{url,enumitem,array,graphicx}
\usepackage[
	colorlinks=true,
	citecolor=black,
	linkcolor=black,
	urlcolor=black]{hyperref}
	
\usepackage[
	sorting=nyt,    % sort by name, year,title
	style=numeric,  % use [1] style citations
	maxnames=5,     % show upto five names
	minnames=3,     % if more than 5 names use the first 3 names and then "et al." for the rest
	firstinits=true % abbreviate author first names
]{biblatex}

\bibliography{fullreferencelist.bib}

\renewbibmacro{in:}{%
 \ifentrytype{article}{}{%
 \printtext{\bibstring{in}\intitlepunct}}
 }

\usepackage{tikz}
\usetikzlibrary{calc,scopes}
\usetikzlibrary{backgrounds}

\tikzset{%
	int/.style={fill=white,opacity=.85,pos=.5, inner sep=.1em,font=\scriptsize},%
	pair/.style={coordinate,shape=circle,draw,scale=.3,fill=white, text=white},%
	solid/.style={coordinate,shape=circle,draw,scale=.3,fill=black},%
}

\newcommand{\cald}{\mathcal{D}}
\newcommand{\calf}{\mathcal{F}}
\newcommand{\calh}{\mathcal{H}}

\newcommand{\frakb}{\mathfrak{b}}
\newcommand{\frakt}{\mathfrak{t}}

\newcommand{\frakB}{\mathfrak{B}}

\newcommand{\fieldstyle}[1]{\mathbb{#1}}
\renewcommand{\AA}{\fieldstyle{A}}

\newcommand{\NN}{\fieldstyle{N}}

\theoremstyle{plain}
\newtheorem{thm}{Theorem}
\newtheorem*{thm*}{Theorem}
\newtheorem{lemma}[thm]{Lemma}

\newtheorem*{cor*}{Corollary}

\theoremstyle{definition}
\newtheorem{dfn}[thm]{Definition}

\newtheorem{ex}[thm]{Example}

\newtheorem{remark}[thm]{Remark}
\newtheorem*{remark*}{Remark}

\theoremstyle{remark}

\newcommand{\IF}{\text{ if }}
\newcommand{\df}[1]{\emph{#1}}

\DeclareMathOperator{\AG}{AG}

\DeclareMathOperator{\cmod}{mod}

\newcommand{\e}{\varepsilon}

\renewcommand{\tilde}{\widetilde}

\newcommand{\alg}{\Lambda}

\renewcommand{\aa}{\alpha}
\newcommand{\bb}{\beta}
\newcommand{\cc}{\gamma}
\newcommand{\dd}{\delta}
\newcommand{\Path}{\Pi}

\newcommand{\te}{\tau}

\newcommand{\xto}[1]{\xrightarrow{\ \ #1\ \ }}
\newcommand{\xlto}[1]{\xleftarrow{\ \ #1\ \ }}

\begin{document}

\title{The AG-invariant for $(m+2)$-angulations}
\author{Lucas David-Roesler}
\date{\today}

\thanks{The author would like to thank Ralf Schiffler and Ben Salisbury for reading 
the very rough drafts of this paper and for their important comments and suggestions.}

\begin{abstract}
In this paper, we study gentle algebras that come from $(m+2)$-angulations of unpunctured 
Riemann surfaces with boundary and marked points. 
We focus on calculating a derived invariant introduced by Avella-Alaminos and Geiss,
generalizing previous work done when $m=1$. In particular, we provide a method for calculating
this invariant based on the the configuration of the arcs in the $(m+2)$-angulation, the 
marked points, and the boundary components. 
\end{abstract}

\maketitle

\section{Introduction}

The derived equivalence classes of $m$-cluster-tilted algebras of type $\AA$ were
determined in \cite{Mur-PRE} using the cycles with full relations in the bound quiver
$(Q,I)$ as an invariant. Recently, in \cite{BG-PRE}, the Hochschild cohomology and an
invariant of Avella-Alaminos and Geiss \cite{AG}, the AG-invariant, were used to
describe all connected algebras derived equivalent to a connected component of an
$m$-cluster-tilted algebra of type $\AA$. Given an algebra $\alg$, both
\cite{BG-PRE,Mur-PRE} make use of a normal form $N_{r,s}$ associated to the bound
quiver $(Q,I)$ of $\alg$. In particular, \cite{BG-PRE} calculates the AG-invariant of
$\alg$ using the normal form quiver $N_{r,s}$. This method generalizes the method for
calculating the AG-invariant for the iterated tilted algebras of type $\AA$ given in
\cite{DS}. We will make use of the description of $m$-cluster-tilted algebra as
arising from the $(m+2)$-angulation of an $(n+1)m+2$-gon $(P,M,T)$.

To state the main theorem we need to define certain sets. Call $M$ the set of marked
points, one for each vertex of the polygon $P$. The set $T$ is the collection of
diagonals in the $(m+2)$-angulation. We identify are particular subset $M_T$ if $M$
consisting of those marked points incident to at least one diagonal in $T$. Denote by
$\frakB$ the set of boundary segments $\frakb$ which are demarcated by the
elements in $M_T$. Finally, set 
\[(a,b)^* = f(x,y) =
\begin{cases}
1 & (x,y) = (a,b),\\
0 & \text{ otherwise}.
\end{cases}\]

\begin{thm*}\label{mainthm}
Let $\alg$ be a connected algebra associated to the $(m+2)$-angulation $(P,M,T)$.  
The AG-invariant of $\alg$ is the function $(a,b)^* + t(0,m+2)^*$ where
	\[ a = \# M_T, \quad \text{and}\quad b = \sum_{\frakB}(m-w(\frakb_i))\]
with $w(\frakb_i) = \#(\frakb_i\cap M\setminus M_T)$ and $t$ is the number of internal
$(m+2)$-gons in $(P,M,T)$. 
\end{thm*}

When $m=1$, this calculation recovers Theorem~4.6 from \cite{DS}. Additionally, in
\cite{DS} a class of algebras called surface algebras is introduced via a process called
admissible cutting of the surface. When the surface is a disc (that is, when we consider a
triangulation of a polygon) this construction produces the iterated tilted algebras of
type $\AA$ with global dimension 2. We will show how to realize these algebras as
$m$-cluster-tilted algebras of type $\AA$. This realization agrees with Corollary~6.6 in
\cite{BG-PRE}.

After considering the disc the next step is to consider those surfaces with multiple
boundary components. When considering triangulations, the immediate benefit is that
cluster-tilted algebras of affine type $\tilde\AA$ arise as triangulations of the
annulus. In \cite{DS}, the calculation of the AG-invariant is done for any surface
with any number of boundary components. Similarly, we extend the main theorem to
those surfaces with more than one boundary component. This work also shows how to
deal with $(m+2)$-angulations that are are not connected.

To calculate the AG-invariant for arbitrary $(m+2)$-angulations $(P,M,T)$, we introduce 
the concept of boundary bridges in section~\ref{general}. We then define a new 
$(m+2)$-angulation $(\tilde P,\tilde M,\tilde T)$ by removing these boundary
bridges. This operation does not disturb the arcs in $T$, so 
$(Q_{\tilde T},I_{\tilde T}) = (Q_T,I_T)$. Using this new surface we can calculate 
the AG-invariant as before. 

\begin{thm*}
Let $(P,M,T)$ be any $(m+2)$-angulation and $\alg$ the corresponding algebra. 
The AG-invariant of $\alg$ is given by
\[t(0,m+2)^* + \sum_{i}(a_i,b_i)^* \] 
where $t$ is number of internal $(m+2)$-gons in $(P,M,T)$, 
$i$ indexes the boundary components of $\tilde P$, and 
\[ a_i = \# \tilde M_T^i, \quad \text{and}\quad b_i= \sum_{\frakb_j\in \tilde\frakB_i}(m-w(\frakb_j)).\]
\end{thm*}

In this theorem $\tilde M_T^i$ and $\frakB_i$ are defined as before but restricted to the 
$i$th boundary component. Note that if there are no boundary bridges in $(P,M,T)$, then
we set $(\tilde P,\tilde M, \tilde T) = (P,M,T)$.

\section{Preliminaries}

\subsection{Gentle algebras}
Let $k$ be an algebraically closed field.
Recall from \cite{AS} that a finite-dimensional algebra $\alg$ is {\em gentle} if it
admits a presentation $\alg=kQ/I$ satisfying the following {\nobreak conditions:}
\begin{itemize}
\item[(G1)] At each point of $Q$ there starts at most two arrows and stops at
  most two arrows.
\item[(G2)] The ideal $I$ is generated by paths of length 2.
\item[(G3)] For each arrow $\bb$ there is at most one arrow $\aa$
  and at most one arrow $\cc$ such that $\aa \bb \in I$
  and $\bb \cc \in I$.
\item[(G4)] For each arrow $\bb$ there is at most one arrow $\aa$
  and at most one arrow $\cc$ such that $\aa \bb \not\in I$
  and $\bb \cc \not\in I$.
\end{itemize}

An algebra $\alg=kQ/I$ where $I$ is generated by paths and $(Q,I)$ satisfies the two
conditions (G1) and (G4) is called a \df{string algebra} (see \cite{BR}), thus every
gentle algebra is a string algebra.

\subsection{The AG-invariant}
We recall from \cite{AG} the definition of the derived invariant of Avella-Alaminos
and Geiss. From this point on called the AG-invariant. Let $\alg$ be a gentle
$k$-algebra with bound quiver $(Q,I)$, $Q=(Q_0,Q_1,s,t)$ where $s,t\colon Q_1\to Q_0$ are
the source and target functions on the arrows.

\begin{dfn}
A \df{permitted path} of $\alg$ is a path $H=\aa_1\aa_2\cdots \aa_n$ in $Q$ which is not in $I$.
We say a permitted path is a \df{non-trivial permitted thread} of $\alg$ if for all arrows
$ \bb\in Q_1$, neither $ \bb H$ nor $H \bb$ is a permitted path. These are the
`maximal' permitted paths of $\alg$. Dual to this, we define the \df{forbidden paths} of
$\alg$ to be a sequence $F= \aa_1\aa_2\cdots \aa_n$ in $Q$ such that $\aa_i\ne \aa_j$ unless
$i=j$, and $\aa_i\aa_{i+1}\in I$, for $i=1,\dots,n-1$. A forbidden path $F$ is a
\df{non-trivial forbidden thread} if for all $ \bb\in Q_1$, neither $ \bb F$ or $F \bb$ is
a forbidden path. We also require \df{trivial permitted} and \df{trivial forbidden
threads}. Let $x\in Q_0$ such that there is at most one arrow starting at $x$ and at most
one arrow ending at $x$. Then the constant path $e_x$ is a trivial permitted thread if
when there are arrows $ \bb, \cc\in Q_1$ such that $s( \cc)=x=t( \bb)$, then $ \bb
\cc\not\in I$. Similarly, $e_x$ is a trivial forbidden thread if when there are arrows $
\bb, \cc\in Q_1$ such that $s( \cc)=x=t( \bb)$, then $ \bb \cc\in I$.
  
Let $\calh$ denote the set of all permitted threads and $\calf$ denote the set of all
forbidden threads.
\end{dfn}

Notice that each arrow in $Q_1$ is both a permitted and a forbidden path. Moreover, the
constant path at each sink and at each source will simultaneously satisfy the definition
for a permitted and a forbidden thread because there are no paths going through $x$.

We fix a choice of functions $\sigma,\e\colon Q_1\to \{-1,1\}$ characterized by the
following conditions.
\begin{enumerate}
\item If $ \aa_1\neq  \aa_2$ are arrows with $s( \aa_1)=s( \aa_2)$, then 
	$\sigma( \aa_1)=-\sigma( \aa_2)$.
\item If $ \aa_1\neq  \aa_2$ are arrows with $t( \aa_1)=t( \aa_2)$, then 
$\e( \aa_1)=-\e( \aa_2)$.
\item If $ \aa, \bb$ are arrows with $s( \bb)=t( \aa)$ and $ \aa \bb\not\in I$, 
then $\sigma( \bb)=-\e( \aa)$.
\end{enumerate}
Note that the functions need not be unique. Given a pair $\sigma$ and $\e$, we can define
another pair $\sigma':=-1\sigma$ and $\e':=-1\e$.

These functions naturally extend to paths in $Q$. Let $\Path = \aa_1\aa_{2}\cdots
\aa_{n-1}\aa_n$ be a path. Then $\sigma(\Path) = \sigma(\aa_1)$ and $\e(\Path)=\e(\aa_n)$.
We can also extend these functions to trivial threads. Let $x,y$ be vertices in $Q_0$,
$h_x$ the trivial permitted thread at $x$, and $p_y$ the trivial forbidden thread at $y$.
Then we set
\begin{align*} 
\sigma(h_x) = -\e(h_x) &= -\sigma(\aa), & \IF& s(\aa)=x, \text{ or}\\
 \sigma(h_x) = -\e(h_x) &= -\e( \bb), & \IF& t( \bb)=x
 \end{align*}
and
\begin{align*}
\sigma(p_y) = \e(p_y)& = -\sigma( \cc), &  \IF& s( \cc) = y, \text{ or}\\
\sigma(p_y) =\e(p_y)&  = -\e(\dd), & \IF& t(\dd) = y ,
\end{align*} 
where $\aa, \bb, \cc,\dd\in Q_1$. Recall that these arrows are unique if they exist.

\begin{dfn}
The AG-invariant $\AG(\alg)$ is defined to be a function depending on the ordered pairs
generated by the following algorithm.
\begin{enumerate}
\item \begin{enumerate}
	\item Begin with a permitted thread of $\alg$, call it $H_0$.
	\item \label{alg:HtoF} To $H_i$ we associate $F_i$, the forbidden thread which ends at 
		$t(H_i)$ and such that  $\e(H_i)=-\e(F_i)$. Define $\varphi(H_i) := F_i$.
	\item \label{alg:FtoH}To $F_i$ we associate $H_{i+1}$, the permitted thread which starts 
		at $s(F_i)$ and such  that $\sigma(F_i)=-\sigma(H_{i+1})$. Define $\psi(F_i):= H_{i+1}$.
	\item Stop when $H_n=H_0$ for some natural number $n$.  Define $m=\sum_{i=1}^n \ell(F_i)$, 
		where $\ell(C)$  is the length (number of arrows) of a path $C$. In this way we obtain the 
		pair $(n,m)$. 
	\end{enumerate}
\item Repeat (1) until all permitted threads of $A$ have occurred.
\item For each oriented cycle in which each pair of consecutive arrows form a relation, we
associate the ordered pair $(0,n)$, where $n$ is the length of the cycle.
\end{enumerate}
We define $\AG(\alg)\colon \NN^2\to \NN$ where $\AG(\alg)(n,m)$ is the number of times the ordered pair 
$(n,m)$ is formed by the above algorithm. 
\end{dfn}

\begin{ex}
Let $(Q,I)$ be the following quiver:
\[\begin{tikzpicture}[scale=.66]
\node[name=1] at (0,0) {$1$};
\node[name=2] at (3,0) {$2$};
\node[name=4] at (1.5,1.5) {$4$};
\node[name=3] at (0,3) {$3$};
\node[name=5] at (3,3) {$5$};
\path[->]
	(1) edge[bend left] node[int] {$\aa_1$} (3)
	(2) edge node[int] {$\aa_2$} (1)
	(3) edge node[int] {$\aa_3$} (4)
	(4) edge node[int] {$\aa_5$} (5)
	    edge node[int] {$\aa_4$} (2);
\path[dashed]
	(3) edge[bend right] (5)
	(4) edge[bend left] (1);
\end{tikzpicture}\]
where $I= \langle \aa_3\aa_5, \aa_4\aa_2\rangle$.
Then the permitted threads are
\[\calh = \{\aa_2\aa_1\aa_3\aa_4,\aa_5,e_1,e_3,e_5\}\]
and the forbidden threads are
\[\calf = \{\aa_4\aa_2,\aa_3\aa_5,\aa_1,e_2,e_5\}.\]
Notice that $e_5$ is both a permitted and forbidden trivial thread.  When necessary, 
we will use the notation $h_x$ and $p_x$ to distinguish when we consider $e_x$ a
permitted or forbidden thread respectively. 

 We can define the 
functions $\sigma$ and $\e$ such that on the threads of $(Q,I)$ we have:
\[
\begin{tabular}{>{$}r<{$}>{$}r<{$}>{$}r<{$}}
\calh & \sigma & \e \\ \hline
\aa_1 &  1 & -1 \\ %
\aa_2 &  1 & -1 \\ %
\aa_3 &  1 &  1 \\ %
\aa_4 & -1 & -1 \\ %
\aa_5 &  1 &  1 \\
\end{tabular}
\hspace{2cm}
\begin{tabular}{>{$}r<{$}>{$}r<{$}>{$}r<{$}}
\calh & \sigma & \e \\ \hline
\aa_2\aa_1\aa_3\aa_4 & 1 & 1 \\
\aa_5 & 1 & 1 \\
h_5 & 1 & -1\\
h_1 & -1 & 1\\
h_3 & -1 & 1\\
\end{tabular}
\hspace{2cm}
\begin{tabular}{>{$}r<{$}>{$}r<{$}>{$}r<{$}}
\calf & \sigma & \e \\ \hline
\aa_4\aa_2 & -1 & -1\\
\aa_3\aa_5 & 1 & 1\\
\aa_1 & 1 & -1\\
p_2 & -1 & -1\\
p_5 & -1 & -1\\
\end{tabular}
\]

Then the calculation of the AG-invariant is given in the following tables:

\[ 
\begin{tabular}{r>{$}c<{$}>{$}c<{$}}
  & H_i & F_i    \\ \hline
0 & h_5 & \aa_3\aa_5 \\
1 & h_3 & \aa_1 \\
2 & h_1 & \aa_4\aa_2 \\
3 & \aa_5 & p_5 \\
4 & h_5 & \\ \hline
  & (4,5) &
\end{tabular}
\hspace{2cm}
\begin{tabular}{r>{$}c<{$}>{$}c<{$}}
  & H_i & F_i    \\ \hline
0 & \aa_2\aa_1\aa_3\aa_4 & p_2  \\
1 & \aa_2\aa_1\aa_3\aa_4 &  \\ \hline
  & (1,0) &\\
\end{tabular}
\]
In this case we have
\[
\AG(\alg)(n,m) = (1,0)^* + (4,5)^*=
\begin{cases}
1 \text{; if } (n,m) = (1,0) \text{ or } (4,5),\\
0 \text{; otherwise}.
\end{cases}
\]
\end{ex}

The algorithm defining $\AG(\alg)$ can be thought of as dictating a walk in the quiver $Q$,
where we move forward on permitted threads and backward on forbidden threads, see~\cite{AG}.

\begin{remark}\label{rmk:bijection}
Note that the steps \eqref{alg:HtoF} and \eqref{alg:FtoH} of this algorithm give two
different bijections $\varphi$ and $\psi$ between the set of permitted threads $\calh$ and
the set of forbidden threads which do not start and end in the same vertex. We will often
refer to the permitted (respectively forbidden) thread ``corresponding'' to a given
forbidden (respectively permitted) thread. This correspondence is referring to the
bijection $\varphi$ (respectively $\psi$).
\end{remark}

\subsection{$m$-cluster-tilted algebras}

Cluster categories were introduced in \cite{BMRRT}. Quickly after, the notion of
cluster-tilted algebras were introduced and studied in \cite{ABCP,ABS,BB,BMR} to name
only a few. This construction was then generalized to $m$-cluster categories and
$m$-cluster-titled algebras in \cite{Th, K05,FPT-PRE} among others. Roughly, the
$m$-cluster category is defined as the orbit category $\cald^b(\cmod
kQ)/\langle\tau^{-1}[m]\rangle$ where $Q$ is an acyclic quiver. The
$m$-cluster-tilted algebras are then defined as the endomorphism algebras of tilting
objects in this category. For a complete description of this development we recommend
\cite{FPT-PRE,Th}. We focus on the combinatorial description of $m$-cluster categories
and the corresponding $m$-cluster-tilted algebras given via $(m+2)$-angulations which has
been studied in \cite{BM1,BT,Mur-PRE}, we will generally adapt the definitions found
in \cite{Mur-PRE}.

Let $P$ be a disc with boundary and $M$ be a finite set contained in $\partial P$. Note that
$(P,M)$ is equivalent to a polygon with $\# M$ edges. It is common to simply take a
polygon, but we prefer the generality of the language of a surface with marked points.

\begin{dfn}
An $m$-allowable diagonal in $(P,M)$ is a chord joining two non-adjacent points in $M$
such that $(P,M)$ is divided into two smaller polygons $P_1$ and $P_2$ which can themselves
be divided into $(m+2)$-gons by non-crossing chords. 
\end{dfn}

\begin{dfn}
The collection $(P,M,T)$ is called an \emph{$(m+2)$-angulation} of $(P,M)$ if $T$ is a maximal 
collection of $m$-allowable diagonals. We denote a $(m+2)$-gon in $(P,M,T)$ by $\triangle$. 
\end{dfn}

In \cite{Mur-PRE}, it is a simple lemma that $(P,M)$ can be divided into an $(m+2)$-angulation
if and only if $\#M\equiv 2\cmod m$.

\begin{dfn}\label{def Qt}
To a $(m+2)$-angulation $(P,M,T)$ we associate a quiver with relations $(Q_T,I_T)$.
The vertices of $Q_T$ are in bijection with the elements of $T$.  For any two vertices
$x$ and $y$ in $Q_T$ we have an arrow $x\to y$ if and only if:
\begin{enumerate}
\item the corresponding $m$-allowable diagonals $\te_x$ and $\te_y$ share a vertex in $(P,M)$,
\item $\te_x$ and $\te_y$ are edges of the same $(m+2)$-gon $\triangle$ in the $(P,M,T)$,
\item $\te_y$ follows $\te_x$ in the counter-clockwise direction (as you walk around the boundary
 of $\triangle$).
\end{enumerate}
\end{dfn}

\begin{ex}\label{ex Qt}
The quiver associated to the $(m+2)$-angulation given in Figure~\ref{fig M and B} is\
\[
1 \xto{\aa_1} 2 \xto{\aa_2} 3 \xlto{\aa_3} 4 \xlto{\aa_4} 5 \xto{\aa_5} 6 \xto{\aa_6} 7
\]
with $I_T = \langle \aa_1\aa_2,\aa_5\aa_6\rangle$. 
\end{ex}

In much of the literature, people choose either the counter-clockwise orientation or the
clockwise orientation. The choice does not affect the final results of the theory  but 
should be carefully noted when doing calculations, choosing to use the clockwise orientation 
will produce $Q_T^{\mathrm{op}}$.  Some of the lemmas in the following section will depend on 
this choice of direction but can easily be restated in the clockwise direction.

\section{Calculating the AG-invariant}
\label{calculation}

\begin{dfn}\label{def ag pieces}
Let $\alg$ be the $m$-cluster-tilted algebra associated to the $(m+2)$-angulation $(P,M,T)$.  
Let
 \[M_T := \{p\in M : p \text{ is incident to } T\},\]
and 
 \[\frakB=\{\frakb_1,\dots,\frakb_r\}\] 
be the pieces of the boundary component such that the endpoints of $\frakb_i$ are 
in $M_T$ and each $\frakb_i$ does not contain any other points of $M_T$.  Further, let 
$w(\frakb_i)$ be the number of marked points on $\frakb_i$ not contained in $M_T$. That is
\[ w(\frakb_i) = \#((M\cap \frakb_i)\setminus M_T).\]
\end{dfn}

\begin{ex}
In Figure~\ref{fig M and B} the set $M_T$ is given by the white marked points.
\begin{figure}
\centering
\begin{tikzpicture}[scale=.33]
\draw (0,0) circle[radius=5cm];
\foreach \a in {0,20,...,340}
 \node[solid,name=n\a] at (\a:5cm) {};

\draw
 (n0) -- (n60) node[int] {$\te_1$}
 (n60) -- (n120) node[int] {$\te_2$}
 (n120) -- (n340) node[int] {$\te_3$}
 (n340) -- (n160) node[int] {$\te_4$}
 (n340) -- (n200) node[int] {$\te_5$}
 (n200) -- (n260) node[int] {$\te_6$}
 (n260) -- (n320) node[int] {$\te_7$};
\foreach \a in {0,60,120,160,340,200,260,320}
	\node[pair] at (\a:5cm) {};
\node[right] at (30:5cm) {$\frakb_1$};
\node[above] at (90:5cm) {$\frakb_2$};
\node[left] at (140:5cm) {$\frakb_3$};
\node[left] at (180:5cm) {$\frakb_4$};
\node[below] at (230:5cm) {$\frakb_5$};
\node[below] at (300:5cm) {$\frakb_6$};
\node[right] at (330:5cm) {$\frakb_7$};
\node[right] at (350:5cm) {$\frakb_8$};
\end{tikzpicture}
\caption{The white points indicate elements of $M_T$.  The boundary segments have 
weights $w(\frakb_1)=w(\frakb_2)=w(\frakb_5)=w(\frakb_6) = 2$, $w(\frakb_3) =  w(\frakb_4) = 1$, 
and $w(\frakb_7)=w(\frakb_8)=0$.}
\label{fig M and B}
\end{figure}
\end{ex}

\begin{lemma}\label{lem a}
Let $\alg$ be the algebra associated to the $(m+2)$-angulation $(P,M,T)$. The permitted
threads of $\alg$ are in bijection with $M_{T}$
\end{lemma}

\begin{proof}
This follows from the definition of $Q_T$. By construction, the arrows of $Q_T$ are given
by the angles of the $(m+2)$-gons in $(P,M,T)$, further arrows $\aa$ and $\bb$ are
composable if and only if the angles defining $\aa$ and $\bb$ are incident to each other at
the same marked point. It follows that to any permitted thread we can associate a marked
point in $M_T$. Conversely, given a marked point in $p\in M_T$, we can associate a sequence
of arrows $H$ defined by the angles incident to $p$. This sequence must define a permitted
thread, since any other arrows that we may consider to compose with $H$ must come from
angles not incident to $p$. Hence the composition is zero by the definition of $I_T$.
Notice that the trivial forbidden threads are given by marked points incident to a unique
edge in $T$.
\end{proof}

\begin{ex}
Applying Lemma~\ref{lem a} to Figure~\ref{fig M and B} and using the labels on the corresponding
quiver given in  Example~\ref{ex Qt} we get the following list of permitted threads written in
counter-clockwise order based on the corresponding marked point:
\[ \aa_4\aa_3, e_1,\aa_1,\aa_2,e_4,\aa_5,\aa_6,e_7.\]
\end{ex}

\begin{lemma}\label{lem b}
Let $\alg$ be the algebra associated to the connected $(m+2)$-angulation $(P,M,T)$. The
forbidden threads of $\alg$ are in bijection with $\frakB$. Further, if $F\in\calf$ is
associated to $\frakb_i\in\frakB$, then $\ell(F) = m - w(\frakb_i)$.
\end{lemma}

\begin{proof}
This proof is similar to the proof of Lemma~\ref{lem a}. By the definition of $I_T$,
the composition $\aa\bb\in I_T$ if and only if $\aa$ and $\bb$ are defined by
neighboring angles of the same $(m+2)$-gon. Hence, the forbidden threads can be
identified with the $(m+2)$-gons of $(P,M,T)$. Additionally, by assumption the
$(m+2)$-angulation is connected, hence each non-internal $(m+2)$-gon contains exactly one
segment from $\frakB$, giving us the identification with elements of $\frakB$. We do
not include the interior $(m+2)$-gons because these, by definition, give rise to an
oriented cycle of relations, hence there is no terminal arrow to define the thread.

Similarly, it is clear from the definition of $I_T$ that, given an $(m+2)$-gon $\triangle$ 
bounded by some $\frakb$, the composition of the arrows defined in $\triangle$ defines
a forbidden thread.

Note that this correspondence also holds for trivial forbidden threads, these threads
correspond to $(m+2)$-gons which contain exactly two points from $M_T$. Such a
$(m+2)$-gon contains a single edge of $T$, say $\tau_i$, which corresponds to a source,
a sink, or the intermediate vertex of a relation in $I_T$. In each of these cases
$e_i$ is a trivial forbidden thread.

Let $F$ be a forbidden thread, $\triangle$ the corresponding $(m+2)$-gon, and $\frakb$
the corresponding edge from $\frakB$. From the first paragraph, we immediately see
that $\ell(F)$ is the number of angles in $\triangle$ constructed from edges in $T$.
We wish to count these angles. There are $(m+2)$ total angles in $\triangle$ coming in
three types: $\ell(F)$ many angles completely constructed by $T$, $w(\frakb)$ many
completely constructed by $\frakb$, and the two angles where $T$ and $\frakb$ meet.
Hence, we have $m+2 = \ell(F)+w(\frakb) + 2$ and we immediately see $\ell(F) = m-
w(\frakb)$, as desired.
\end{proof}

Recall from Remark~\ref{rmk:bijection}, the steps \eqref{alg:HtoF} and \eqref{alg:FtoH} of
the AG algorithm give two different bijections $\varphi$ and $\psi$ between the set of
permitted threads $\calh$ and the set of forbidden threads which do not start and end in
the same vertex. Throughout the proofs of the following lemmas we will use $p\in M_T$ and 
$\frakb\in\frakB$ to both represent the elements of each respective set but also the 
corresponding permitted or forbidden thread in $(Q_T,I_T)$.

\begin{lemma}\label{lem 1}
Let $p\in M_T$ and let $p$ also denote the corresponding permitted thread, then the
forbidden thread $\varphi(p)$ is given by the edge $\frakb\in\frakB$ incident to and
following $p$ in the counter-clockwise direction.
\end{lemma}

\begin{proof}
The sequence of edges incident to $p$ can end in two ways: (1) bounding an
$(m+2)$-gon $\triangle$ incident to exactly two points of $M_T$ or (2) bounding $\triangle$
which is incident to more than two points of of $M_T$. In both cases, $\triangle$ contains a
boundary segment from $\frakB$, let $\frakb$ be this segment. In the first case, the
boundary segment $\frakb$ corresponds to a trivial forbidden thread, we have one of the
following figures:
\[
%\begin{tikzpicture}[scale=.33,baseline=30]
%\draw (0,0) arc (-45:0:5cm) node[name=p,solid,pos=.33] {}
%node[name=q,solid,pos=1] {} ;
%\draw[dotted] (q) arc (0:30:5cm) node[name=r,solid,pos=1] {};
%\draw (r) arc (30:60:5cm)  node[name=s,solid,pos=.9] {} ;
%\draw[dotted] (s) arc (60:80:5cm) node[name=coord,pos=1] {};
%\draw (coord) arc (80:100:5cm)  node[name=t,solid,pos=.2] {};
%\path (p) edge (s) edge (t);
%\node[below right] at (p) {$p$};
%\node[above right] at (q) {$\frakb$};
%\node[above right] at (t) {$\frakb'$};
%\end{tikzpicture}
\includegraphics{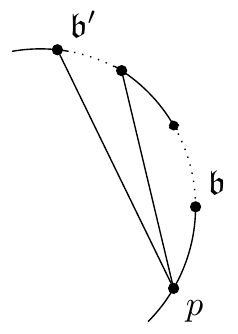}
\qquad  \raisebox{1.2cm}{or}\qquad
%\begin{tikzpicture}[scale=.31,baseline=-7]
%\node[name=p,solid] at (-30:5cm) {}; 
%\draw (p) arc (-30:0:5cm)  node[name=q,solid,pos=1] {};
%\draw[dotted] (p) arc (-30:-55:5cm) node[name=o,solid,pos=1] {} ;
%\draw (o) arc (-55:-65:5cm) ;
%\draw[dotted] (q) arc (0:30:5cm) node[name=r,solid,pos=1] {};
%\draw (r) arc (30:90:5cm)  node[name=s,solid,pos=.33] {} node[name=t,solid,pos=.66] {};
%\node[below right] at (p) {$p$};
%\node[above right] at (q) {$\frakb$};
%\path (p) edge (s) (s) edge (o);
%\end{tikzpicture}
\includegraphics{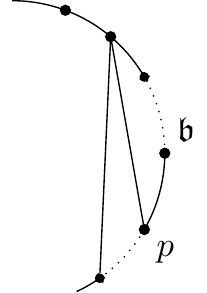}
\qquad \raisebox{1.2cm}{or}\qquad
%\begin{tikzpicture}[scale=.33,baseline=30]
%\draw (0,0) arc (-45:0:5cm) node[name=p,solid,pos=.33] {}
%node[name=q,solid,pos=1] {} ;
%\draw[dotted] (q) arc (0:30:5cm) node[name=r,solid,pos=1] {};
%\draw (r) arc (30:90:5cm)  node[name=s,solid,pos=.33] {} node[name=t,solid,pos=.66] {};
%\path (p) edge (s);
%\draw (p) -- ++(175:4cm);
%\draw (s) -- ++(200:4cm);
%\node[below right] at (p) {$p$};
%\node[above right] at (q) {$\frakb$};
%\end{tikzpicture}
\includegraphics{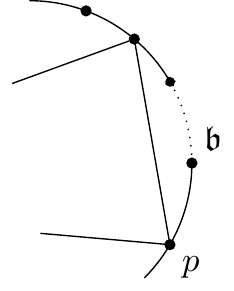}
\]
corresponding to a sink, source or neither in $Q_T$. In the first figure there are
two forbidden threads $\frakb$ and $\frakb'$ ending at $t(p)$. By the definition of
$\e$ we must have $\e(\frakb') = \e(p)$ because $\frakb$ corresponds to the final
arrow of $p$. Similarly, $\e(\frakb) = -\e(p)$, hence $\varphi(p)=\frakb$. In the
last two cases the only possible forbidden thread ending at $t(p)$ is $\frakb$. It is
a simple check that $\frakb$ must satisfy the compatibility condition on $\e$ used in
step 1(b) defining the AG-invariant. It follows that $\varphi(p)=\frakb$.

If the ending $(m+2)$-gon $\triangle$ is of type (2), that is, if $\triangle$ contains more
that two points of $M_T$, then there is a unique choice for the forbidden thread 
$\varphi(p)$.  We have the following figure
\[
%\begin{tikzpicture}[scale=.33,baseline=10]
%\foreach \a in {270,325,40,90}
%	\node[name=n\a,solid] at (\a:5cm) {};
%	
%\draw (n90) arc (90:105:5cm)node[pos=1,name=a] {}
%      (n90) arc (90:75:5cm) node[pos=1,name=b] {}
%      
%      (n40) arc (40:55:5cm) node[pos=1,name=c] {}
%      (n40) arc (40:25:5cm) node[pos=1,name=d] {}
%      
%      (n325) arc (325:340:5cm) node[pos=1,name=e] {}
%      (n325) arc (325:310:5cm) node[pos=1,name=f] {}
%      
%      (n270) arc (270:285:5cm) node[pos=1,name=g] {}
%      (n270) arc (270:255:5cm) node[pos=1,name=h] {};
%
%\draw[dotted] (b) edge (c) (d) edge (e) (f) edge (g);
%
%\path 
% (n90) edge node[int] {$\te_x$} (n270) 
% (n90) edge[dotted] node[int] {$\te_y$} (n40) 
% 	(n90) edge ($(n90)!.33!(n40)$)  
% 	(n40) edge ($(n40)!.33!(n90)$)
% (n40) edge node[int] {$\tilde\te$} (n325);
%\draw (n90) -- ++(225:3cm);
%\node[below] at (n270) {$p$};
%\node[right=.33cm] at (g) {$\frakb$};
%\node[above left=.25cm] at (f) {$\triangle$};
%\end{tikzpicture}
\includegraphics{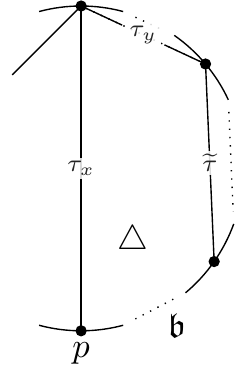}
\qquad \raisebox{1.75cm}{or}\qquad
%\begin{tikzpicture}[scale=.33,baseline=10]
%\foreach \a in {270,325,40,90}
%	\node[name=n\a,solid] at (\a:5cm) {};
%	
%\draw (n90) arc (90:105:5cm)node[pos=1,name=a] {}
%      (n90) arc (90:75:5cm) node[pos=1,name=b] {}
%      
%      (n40) arc (40:55:5cm) node[pos=1,name=c] {}
%      (n40) arc (40:25:5cm) node[pos=1,name=d] {}
%      
%      (n325) arc (325:340:5cm) node[pos=1,name=e] {}
%      (n325) arc (325:310:5cm) node[pos=1,name=f] {}
%      
%      (n270) arc (270:285:5cm) node[pos=1,name=g] {}
%      (n270) arc (270:255:5cm) node[pos=1,name=h] {};
%
%\draw[dotted] (b) edge (c) (d) edge (e) (f) edge (g);
%
%\path 
% (n90) edge node[int] {$\te_x$} (n270) 
% (n90) edge[dotted] node[int] {$\te_y$} (n40) 
% 	(n90) edge ($(n90)!.33!(n40)$)  
% 	(n40) edge ($(n40)!.33!(n90)$)
% (n40) edge node[int] {$\tilde\te$} (n325);
%\draw (n270) -- ++(135:3cm);
%\node[below] at (n270) {$p$};
%\node[right=.33cm] at (g) {$\frakb$};
%\node[above left=.25cm] at (f) {$\triangle$};
%\end{tikzpicture}
\includegraphics{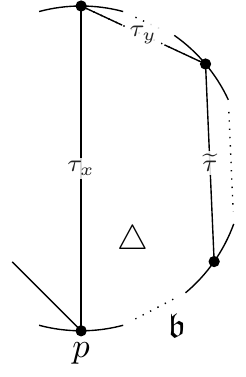}
\]
where we $\tilde\te$ the other required edges of $\triangle$. We allow that $\tilde\te$ does 
not exist, so $w(\frakb) = m-1$.  If $p$ is a trivial thread, then the only forbidden thread
ending at $\tau_x$ is $\frakb$ and it follows immediately that $\e(p)=-\e(\frakb)$ as required.
On the other hand if $p$ is not trivial, then the final arrow $\aa$ of the thread $p$ is also a 
forbidden thread (of length 1). Let $\bb$ denote the final arrow of $\frakb$, which is formed by 
the angle between $\te_y$ and $\te_x$. From the definition of $\e$, we must have
$\e(p) = \e(\aa)=-\e(\bb) = -\e(\bb)$, as desired. Hence $\varphi(p)=\frakb$.
\end{proof}

\begin{lemma}\label{lem 2}
Let $\frakb\in\frakB$ and let $\frakb$ also denote the corresponding forbidden thread,
then the permitted thread $\psi(\frakb)$ is the marked point $p\in M_T$ incident to and
following $\frakb$ in the counter-clockwise direction.
\end{lemma}

\begin{proof}
Let $p$ be as in the statement, further let $\te$ be the edge incident to $p$ and bounding
the $(m+2)$-gon containing $\frakb$.  

As in the previous lemma, we consider two cases. First assume that $w(\frakb) = m$,
so that $\frakb$ is a trivial forbidden thread, call the corresponding vertex
corresponds to $\te$. This has two sub-cases. If $\te$ is not the source of any
arrows, then $p$ is the trivial permitted thread at $\te$ and the definition of
$\sigma$ immediately implies that $\sigma(\frakb) = -\sigma(p)$, hence
$\psi(\frakb)=p$. The second sub-case, if $x$ is the source of an arrow, this arrow
must be unique and formed by an angle incident to $p$. Call this arrow $\aa$, we must
have $\sigma(p) = \sigma(\aa) = -\sigma(\frakb)$.

Now assume that $w(\frakb)<m$, so $\frakb$ is a non-trivial forbidden thread. Let $\bb$ denote
the initial arrow of this thread, note that $s(\bb) = \te$.  In this case, $p$
may represent either a trivial or a non-trivial permitted thread. In both cases, we have
$\sigma(\frakb)= \sigma(\bb) = -\sigma(p)$, hence $\psi(\frakb)=p$. 
\qedhere

\begin{figure}
\centering
\includegraphics{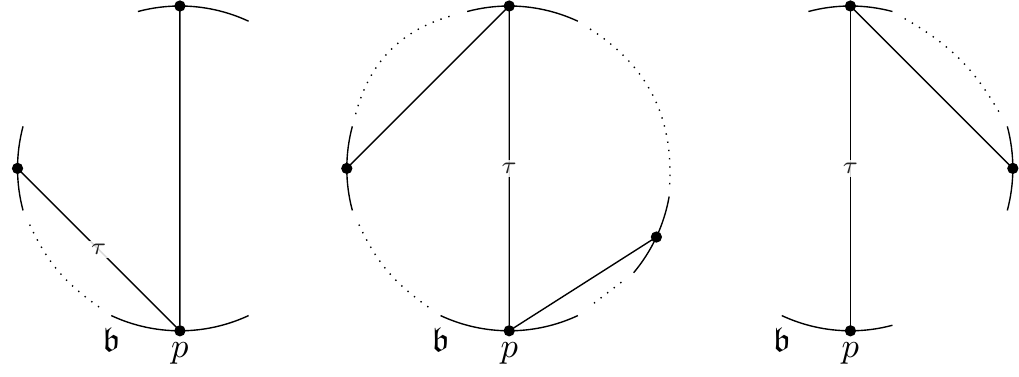}
\caption{Possible configurations for $\te$ in Lemma~\ref{lem 2}.}\
\label{fig lem 2}
\end{figure}

\end{proof}

\begin{thm}\label{main thm}
Let $\alg$ be the algebra associated to the connected $(m+2)$-angulation $(P,M,T)$. The
AG-invariant of $\alg$ is $(a,b)^* + t(0,m+2)^*$ where $t$ is number of internal
$(m+2)$-gons in $(P,M,T)$ and 
	\[ a = \# M_T, \quad \text{and}\quad b= \sum_{\frakB}(m-w(\frakb_i)).\]
\end{thm}
\begin{proof}
This follows immediately from Lemma~\ref{lem a},~\ref{lem b}, \ref{lem 1}, and \ref{lem 2}
and the definition of the AG algorithm.
\end{proof}

\section{Surface algebras as $m$-cluster-tilted}
\label{TitledAsMClusterTilted}

For brevity we omit the definition of surface algebras given in \cite{DS}, hence, we will
not discuss the concept of admissible cuts, instead, we define surface algebras as
algebras arising from particular partial triangulations of surfaces. Further, we will
focus on the case when the surface is a disc. The resulting algebra is iterated tilted of
type $\AA_n$ with global dimension at most 2. The definition we give could easily be extended 
to other type, but it will not be needed. 

\begin{dfn}
Let $(P,M)$ be a disc with marked points in the boundary. Fix a partial
triangulation $T$ such that the non-triangular components are squares containing
exactly one edge in the boundary. As for $(m+2)$-angulations, we define the bound
quiver $(Q_T,I_T)$ where $(Q_T)_0$ is indexed by the edges in $T$ and there is an
arrow $\aa\colon i\to j$ if $\te_i$ and $\te_j$ form an angle in a triangle (or square) in
$T$ and $j$ follows $i$ in the counter-clockwise direction. As before, we say the
arrow $\aa$ lives in the triangle (resp. square) that $\te_i$ and $\te_j$ bound. We
can then define the ideal $I_T$ by setting $\aa\bb\in I_T$ if $\aa$ and $\bb$ live in
the same square.
\end{dfn}

\begin{ex}\label{ex partial}
The quiver associated to the partial triangulation given in Figure~\ref{fig partial} is\
\[
1\xlto{\aa_1} 2\xto{\aa_2} 3\xto{\aa_3}4
\]
with $I_T = \langle \aa_2\aa_3\rangle$. 

\begin{figure}
\centering
\begin{tikzpicture}[scale=.33]
\draw (0,0) circle[radius=5cm];
\foreach \a in {0,45,...,315}
 \node[solid,name=n\a] at (\a:5cm) {};

\draw
 (n0) -- (n90) node[int] {$\te_1$}
 (n0) -- (n135) node[int] {$\te_2$}
 (n0) -- (n270) node[int] {$\te_3$}
 (n270) -- (n180) node[int] {$\te_4$};
\end{tikzpicture}
\caption{A partial triangulation of the disc.}
\label{fig partial}
\end{figure}
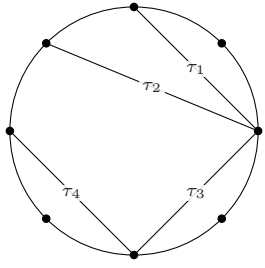
\end{ex}

Given an iterated tilted algebra $\alg$ of type $\AA$ defined via a partial 
triangulation, we will show that $\alg$ is $m$-cluster-tilted for any $m>1$
by realizing it as an $(m+2)$-angulation of the disc. Further, the calculation
of the AG-invariant independent of the choice of $m$. 

Let $(P,M,T)$ be the partial triangulation of $\alg$ and define the sets $M_T$ and
$\frakB$ as in section~\ref{calculation}. For $m>1$, we construct a $(m+2)$-angulation
from $(P,M,T)$ as follows. To each edge $\frakb\in\frakB$ we add the following number
of marked points
\begin{equation}\label{eq new points}\begin{cases}
m-2 & \text{if } \frakb \text{ bounds a square},\\
m-1 & \text{otherwise}.\\
\end{cases}\end{equation}
let $M'$ be the new marked points, we define an $(m+2)$-angulation $(P,M\cup M',T)$ 
where $(P,M,T)$ is the original partial triangulation.  Because we are not creating
any new edges or angles, the quiver of $(P,M\cup M',T)$ is exactly the quiver $Q_T$
associated to $(P,M,T)$. 
Notice that the total number of points for each component polygon of $T$ will be
$4 + m-2 = m+2$ or  $3 + m - 1 = m+2$, hence this process does indeed produce an
$(m+2)$-angulation.  Further, by the construction of $(Q_T,I_T)$,
it immediately follows from the definition given in \cite{Mur-PRE} that $\alg$ is 
$m$-cluster-tilted of type $\AA_n$. 

We note that recent work in \cite{FPT-PRE} has shown how to construct $m$-cluster-tilted
algebras from iterated tilted algebras with global dimension at most $m+1$ via relation
extensions, extending work that was done for $m=1$ in \cite{ABS}. The realization we have
constructed above will only create $m$-cluster-tilted algebras with global dimension at
most 2. It should also be remarked the that above construction verifies a special case of
Corollary~6.6(b) in \cite{BG-PRE}.  We formalize the above discussion in the following
theorem.

\begin{thm}
The iterated tilted algebras of type $\AA$ with global dimension at most 2 are
$m$-cluster-tilted algebras for $m>1$.
\end{thm}

\begin{ex} In Example~\ref{ex partial} we
associated the following quiver to the partial triangulation given in Figure~\ref{fig partial}
\[
1\xlto{\aa_1} 2\xto{\aa_2} 3\xto{\aa_3}4
\]
with $I_T = \langle \aa_2\aa_3\rangle$.  This quiver also corresponds to the $3+2$-angulation 
given in Figure~\ref{fig partial extended} which is constructed using equation~\eqref{eq new points}.

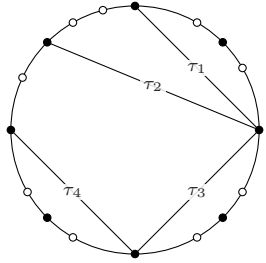
\begin{figure}
\centering
\begin{tikzpicture}[scale=.33]
\draw (0,0) circle[radius=5cm];
\foreach \a in {0,45,...,315}
 \node[solid,name=n\a] at (\a:5cm) {};

\draw
 (n0) -- (n90) node[int] {$\te_1$}
 (n0) -- (n135) node[int] {$\te_2$}
 (n0) -- (n270) node[int] {$\te_3$}
 (n270) -- (n180) node[int] {$\te_4$};
 
\foreach \a in {30,60,105,120,155,210,240,300,330}
 \node[solid,name=n\a,fill=white] at (\a:5cm) {};
 
\end{tikzpicture}
\caption{The $m=3$ version of Figure~\ref{fig partial}.  The new 
marked points determined by equation~\eqref{eq new points} are white.}
\label{fig partial extended}
\end{figure}
\end{ex}

\section{Other surfaces}\label{general}

We begin with an example to demonstrate that the above work can not directly generalize to 
other surfaces.  In the previous sections we have restricted our work to $P$ a disc, we now
consider the annulus.  When $m=1$,  triangulations of the annulus correspond to cluster-tilted
algebras of affine Dynkin type $\tilde\AA$, hence this is a natural next step after 
type $\AA$.  Consider the following $4$-angulation of the 
annulus:
\[\begin{tikzpicture}[scale=.33]
\draw (0,0) circle[radius=5cm];
\filldraw[fill=gray!40] (0,0) circle[radius=1cm];

\foreach \a in {0,90,180,270}
	\node[solid,name=o\a] at (\a:5cm) {};
\foreach \a in {90,270}
	\node[solid,name=i\a] at (\a:1cm) {};
	
\draw
 (o0) ..controls +(115:2cm) and +(0:2cm) .. (90:3cm)
      ..controls +(180:1.5cm) and +(90:1.5cm) ..(180:3cm) node[int] {$\te_2$}
      ..controls +(270:2cm) and +(125:2cm) .. (o270) 
 (o270) ..controls +(45:2cm) and +(270:1cm) .. (0:2cm) node[int] {$\te_3$}
 	   ..controls +(90:1cm) and +(25:1cm) .. (i90)
 (o0) ..controls +(135:2cm) and +(0:1cm) .. (90:2.25cm) node[int,pos=.9] {$\te_1$}
      ..controls +(180:1.25cm) and +(90:1.25cm) .. (180:2.25cm)
      ..controls +(270:1.2cm) and +(200:1cm) .. (i270);
\node[left] at (180:.6cm) {$\frakb_1$};
\node[right] at (-30:5cm) {$\frakb_2$};
\end{tikzpicture}\]
We hope that this would correspond to 2-cluster tilted algebra of type
$\tilde\AA$. If we use same rule as given in Defintion~\ref{def Qt}, then the
corresponding quiver with relations ($Q_T,I_T)$ is
\[ 1 \xrightarrow{\ \ \aa\ \ } 2 \xrightarrow{\ \ \bb\ \ } 3\]
with $I_T = \langle \aa\bb\rangle$. This is an iterated tilted algebra of type $\AA_3$.
Further, we can realize this quiver coming from the following $(2+2)$-angulation 
of the disc
\[\begin{tikzpicture}[scale=.33]
\draw (0,0) circle[radius=5cm];
\foreach \a in {45,135,225,315}
	\node[solid,name=n\a] at (\a:5cm) {};
\foreach \a in {20,-20,160,200,250,290}
	\node[solid] at (\a:5cm) {};
\draw (n45) -- (n315) -- (n225) -- (n135);
\end{tikzpicture}\]
Applying Theorem~\ref{main thm}, tells us that the AG-invariant of $(Q_T,I_T)$ is
$(4,2)^*$.  However, in the spirit of extending Theorem~4.6 from \cite{DS}, in the 
annulus we should apply the main theorem to both boundary components to produce 
$(2,2)^*+(2,4)^*$, clearly the incorrect function.

Let $(P,M,T)$ be an $(m+2)$-angulation of an arbitrary surface $P$ with $M_T$ and
$\frakB$  the sets  given in Definition~\ref{def ag pieces}. The primary issue with
the above example is that the $4$-gon bounded by $\te_1$ and $\te_3$ (but not
$\te_2$) contains more than one boundary segment from $\frakB$. This
$4$-gon encodes information about a sink and a source in the corresponding quiver
which are are forbidden threads. In Lemmas~\ref{lem a} and \ref{lem b} we needed a
clear bijection between the boundary components $\frakB$ and forbidden threads, this
bijection and subsequently the lemmas clearly fails in this case. Even though there
are exactly two boundary components and two threads, it is not clear which
boundary component should correspond to which thread. It is not hard to
see that, for $(m+2)$-angulations of surfaces with an arbitrary number of boundary
components, these lemmas will extend to those $(m+2)$-angulations such that each
$(m+2)$-gon contains at most one element from $\frakB$. In fact, it is also seen that
all of Section~\ref{calculation} will extend to these $(m+2)$-angulations. To make this
concrete, we introduce the following definitions.

\begin{dfn}\label{def B segs}
Let $P$ be a surface with a non-empty boundary with any number of boundary components,
$M$ a set of points in the boundary with at least one point in each component, and $T$
a collection of arcs contained in the interior of $P$ with endpoints in $M$. We say 
that $(P,M,T)$ is a  $(m+2)$-angulation if it subdivides the interior of $P$ into 
$(m+2)$-gons.  As in Section~\ref{calculation}, we set 
\[M_T := \{p\in M : p \text{ is incident to } T\},\]
and 
 \[\frakB=\{\frakb_1,\dots,\frakb_r\}.\] 
the pieces of the boundary such that the endpoints of $\frakb_i$ are 
in $M_T$ and each $\frakb_i$ does not contain any other points of $M_T$.
To distinguish elements in a particular boundary component, $M_T^{i}$ is the set of
marked points in the $i$th boundary component. Similarly, let $\frakB_i$ be the 
elements of $\frakB$ from the $i$th boundary component. The weight of a 
boundary segment, $w(\frakb)$, is defined as before 
(see Definition~\ref{def ag pieces}).
\end{dfn}

\begin{dfn}
Let $(P,M,T)$ be an $(m+2)$-angulation of a surface $P$ (which may have more than one
boundary component). We say $(P,M,T)$ is a \df{non-degenerate} $(m+2)$-angulation if
each $(m+2)$-gon in $(P,M,T)$ contains at most one element of $\frakB$. All other 
$(m+2)$-angulations are called \df{degenerate}.
\end{dfn}

Note that the non-degnerate $(m+2)$-angulations inherently give rise to connected
quivers. The definition of degenerate includes the $(m+2)$-angulations which are 
not connected. 

With these definitions, Theorem~\ref{main thm} can immediately be extended as follows. 
\begin{thm}\label{main thm 2}
If $(P,M,T)$ is a non-degenerate $(m+2)$-angulation and $\alg$ the corresponding 
algebra, then the AG-invariant of $\alg$ is given by 
\[t(0,m+2)^* + \sum_{i>0}(a_i,b_i)^* \] 
where $t$ is number of internal $(m+2)$-gons in $(P,M,T)$, $i$ indexes the boundary components,
and 
\[ a_i = \# M_T^i, \quad \text{and}\quad b_i= \sum_{\frakb_j\in\frakB_i}(m-w(\frakb_j)).\]
\qed
\end{thm}

In the remainder of this section we will calculate the AG-invariant for degenerate
$(m+2)$-angulations. We will do this by a process we call 
\df{bridging boundary components}.  In this process we will introduce new boundary 
segments connecting distinct boundary components of $P$.  This will decrease the number
of boundary components but will be defined so that the set $\frakB$ is extended in 
such a way that the bijection of Lemma~\ref{lem b} is true. In general, we are modifying
the surface so that the resulting $(m+2)$-angulation is no longer degenerate. 

\begin{dfn}\label{def bridge}
Let $\triangle$ be an $(m+2)$-gon which is bounded by $\{\frakb_1,\dots,\frakb_s\}\subset\frakB$
with $s>1$ and $\{\frakt_1,\dots,\frakt_r\}$ the boundary segments of $\triangle$ composed 
of arcs from $T$ with $\frakt_i$ incident to $\frakb_{i-1}$ and $\frakb_{i+1}$ 
(indices are considered modulo $r$).  In the example opening this section, $\frakt_1=\te_1$ and
$\frakt_2=\te_3$. We define the boundary bridge in $\triangle$ as follows.  
\begin{enumerate}
\item For each segment $\frakt_i$, let $\delta_i'$ and $\delta_i''$ be two curves on
$\frakb_{i-1}$ and $\frakb_{i+1}$ respectively starting at the endpoints of
$\frakt_i$ and denote by $v_i'$ and $v_i''$ their respective endpoints. Moreover, we
choose $\delta_i'$, $\delta_i''$ short enough such that $v_i'$ and $v_i''$ are not in
$M$ and no point of $M$ (other than the endpoints of $\frakt$) lie on the curves
$\delta_i'$, $\delta_i''$.
\item Let $\frakt_i'$ denote the arc (up to homotopy) in the interior of $P$ connecting $v_i'$ and $v_i''$.
This results in a new polygon $\triangle_i$ bounded by $\frakt_i$, $\delta_i'$, $\delta_i''$
and $\frakt_i'$.  
\item Add the appropriate number of marked points to $\frakt_i'$ so that  $\triangle_i'$ is a
$(m+2)$-gon. (Recall that we are assuming that $m>1$, so there is zero or more points that need 
to be added, we never need to remove points.)
\item We refer to the complement of the $\triangle_i'$s as the \df{boundary bridge} in 
$\triangle$.  Note that this includes the pieces of $\frakb_i$ not in $\delta_i'$ and $\delta_i''$ 
but does not include the arcs $\frakt'$.
\end{enumerate}
\end{dfn}

\begin{remark}\label{remark T}
Notice that this construction does not affect the set $T$ in any way. In particular, if
$\te_j$ and $\te_k$ are incident before creating the boundary bridges, then they are still
incident after the bridges are constructed.  
\end{remark}

\begin{ex}
Consider the $(4+2)$-angulation of the (genus 0) surface with 4 boundary components
in Figure~\ref{fig 4 cpts}. It is not difficult to see that the corresponding algebra this 
can also be realized from a non-degenerate triangulation of a disc. In this figure, this fact
is suggested by by the fact that the boundary bridges connect all of the original boundary 
components. If we cut out the boundary bridges, there will be a single boundary component. The
corresponding triangulation is then created by reducing the number or marked points until each
$(m+2)$-gon becomes a triangle. Notice that when we cut out the bridges, two of the original 
marked points are removed, so this operation while preserving $T$, does not preserve $M$ as a 
subset of the new marked point set.
\begin{figure}
%\begin{tikzpicture}[scale=.5]
%\draw[thick] (0,0) circle[radius=7] ;
%\filldraw[fill=gray!30,thick] (0:2.5cm) circle[radius=1cm] node[name=A] {};
%\filldraw[fill=gray!30,thick] (120:2.5cm) circle[radius=1cm] node[name=B] {B};
%\filldraw[fill=gray!30,thick] (240:2.5cm) circle[radius=1cm] node[name=C] {C};
%\path[thick] (A) 
%	+(30:1cm) node[solid,name=A1] {}
%	+(90:1cm) node[solid,name=A2] {}
%	+(270:1cm) node[solid,name=A3] {}
%	+(330:1cm) node[solid,name=A4] {};
%
%\path[thick] (B) 
%	+(66:1cm) node[solid,name=B1] {}
%	+(125:1cm) node[solid,name=B2] {}
%	+(210:1cm) node[solid,name=B3] {}
%	+(270:1cm) node[solid,name=B4] {};
%\path[thick] (C) 
%	+(100:1cm) node[solid,name=C1] {}
%	+(210:1cm) node[solid,name=C2] {}
%	+(290:1cm) node[solid,name=C3] {}
%	+(0:1cm) node[solid,name=C4] {};
%\path[thick] (0,0)
%	+(90:7cm) node[solid,name=O1] {}
%	+(180:7cm) node[solid,name=O2] {}
%	+(260:7cm) node[solid,name=O3] {}
%	+(320:7cm) node[solid,name=O4] {};
%
%\draw 
%	(B4) -- (C1)
%	(C4) --(A3)
%	(C1) -- (O2)
%	(B1) -- (O1)
%	(B1) ..controls +(30:2cm) and +(120:1cm) .. (A2) 
%	 ..controls +(30:4cm) and +(120:2cm) .. (O4)
%	 ..controls +(150:2cm) and +(340:2cm) .. (C4)
%	 ..controls +(280:2cm) and +(80:1cm) .. (O3)
%	 ..controls +(140:2cm) and +(290:2cm) .. (O2);
%\foreach \a in {15,45,200,215,230,245,260, 270,285,295,310}
%	\node[solid] at (\a:7cm) {};
%\end{tikzpicture}

\begin{tikzpicture}[scale=.4]
\draw (0,0) circle[radius=7] ;
\filldraw[fill=gray!30] (0:2.5cm) circle[radius=1cm] node[name=A] {};
\filldraw[fill=gray!30] (120:2.5cm) circle[radius=1cm] node[name=B] {};
\filldraw[fill=gray!30] (240:2.5cm) circle[radius=1cm] node[name=C] {};
\path (A) 
	+(30:1cm) node[solid,name=A1] {}
	+(90:1cm) node[solid,name=A2] {}
	+(270:1cm) node[solid,name=A3] {}
	+(330:1cm) node[solid,name=A4] {}
	+(120:1cm) node[pair,name=A5] {}
	+(220:1cm) node[pair,name=A6] {};

\path (B) 
	+(66:1cm) node[solid,name=B1] {}
	+(125:1cm) node[solid,name=B2] {}
	+(210:1cm) node[solid,name=B3] {}
	+(270:1cm) node[solid,name=B4] {}
	+(36:1cm) node[pair,name=B5] {}
	+(300:1cm) node[pair,name=B6] {}
	+(95.5:1cm) node[pair,name=B7] {}
	+(240:1cm) node[pair,name=B8] {};
\path (C) 
	+(100:1cm) node[solid,name=C1] {}
	+(210:1cm) node[solid,name=C2] {}
	+(290:1cm) node[solid,name=C3] {}
	+(0:1cm) node[solid,name=C4] {}
	+(30:1cm) node[pair,name=C5] {}
	+(60:1cm) node[pair,name=C6] {};
\path (0,0)
	+(90:7cm) node[solid,name=O1] {}
	+(180:7cm) node[solid,name=O2] {}
	+(260:7cm) node[solid,name=O3] {}
	+(320:7cm) node[solid,name=O4] {}
	+(100:7cm) node[pair,name=O5] {}
	+(170:7cm) node[pair,name=O6] {};

\draw 
	(B4) -- (C1)
	(C4) --(A3)
	(C1) -- (O2)
	(B1) -- (O1)
	(B1) ..controls +(30:2cm) and +(120:1cm) .. (A2) 
	 ..controls +(30:4cm) and +(120:2cm) .. (O4)
	 ..controls +(150:2cm) and +(340:2cm) .. (C4)
	 ..controls +(280:2cm) and +(80:1cm) .. (O3)
	 ..controls +(140:2cm) and +(290:2cm) .. (O2);
\begin{scope}[on background layer]
\filldraw[fill=gray!60!red] 
  (O5.center) -- (B7.center) node[pair,pos=.33]{} node[pair,pos=.66] {}
  arc[start angle=95.5, end angle=240,radius=1cm] 
  ..controls +(270:3cm) and +(0:2cm) .. (O6.center) 
  		node[pair,pos=.33]{} node[pair,pos=.66] {}
  arc[start angle=170,delta angle=-70,radius=7cm]
  (B6.center) -- (C6.center) node[pair,pos=.33]{} node[pair,pos=.66] {} 
  	arc[start angle=60, delta angle=-30,radius=1cm] 
  	-- (A6.center) node[pair,pos=.33]{} node[pair,pos=.66] {}
  	arc[start angle=220, delta angle=-100,radius=1cm] 
  	..controls +(100:1cm) and +(30:1.2cm).. (B5.center) 
  		node[pair,pos=.33]{} node[pair,pos=.66] {}
  	arc[start angle=36,delta angle=-96,radius=1cm];

\end{scope}

\foreach \a in {15,45,200,215,230,245,260, 270,285,295,310}
	\node[solid] at (\a:7cm) {};

\end{tikzpicture}
\caption{The $(4+2)$-angulation is degenerate, with two 6-gons containing boundary bridges.
The white marked points are those added while constructing the bridges and the dark
red-grey area is the interior of the bridges.}
\label{fig 4 cpts}
\end{figure}
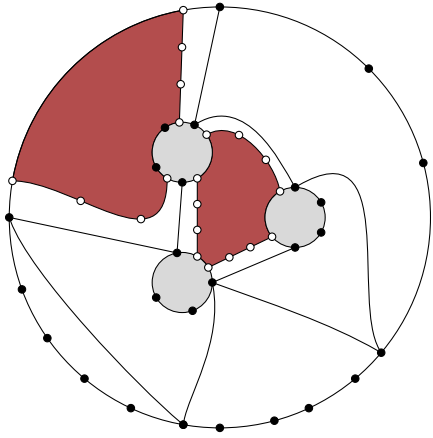

\end{ex}

\begin{dfn}
Let $(P,M,T)$ be a degenerate $(m+2)$-angultion. Define $(\tilde P,\tilde M,\tilde T)$ to be the
$(m+2)$-angulation constructed by removing all boundary bridges from $(P,M,T)$.  For compactness of 
notation we use the following convention that $(\tilde P, \tilde M, \tilde T) = (P,M,T)$ when the 
$(m+2)$-angulation is non-degenerate, in this case there are no boundary bridges and hence nothing 
to remove. 
\end{dfn}

Notice that if $T$ is connected, then $\tilde P$ will be a connected surface. On the other hand,
if $T$ is not connected, then removing the boundary bridges will result in a surface
which is not connected. In either case, the removal of the boundary bridges has a significant
impact on the number of boundary components and the set $\frakB$.  Let $\tilde\frakB$ denote the 
set of boundary segments with endpoints in $\tilde M_T$, as in Definition~\ref{def B segs}.
The set $\tilde\frakB$ consists of the segments in $\frakB$ plus the segments 
$\delta_i'\frakt_i'\delta_i''$ minus the the $\frakb_i$ involved in constructing the boundary 
bridges.  By construction of the bridges, each $\triangle_i'$ will contain exactly one element of 
$\tilde\frakB$; further, by definition the $(m+2)$-gons of $(P,M,T)$ that do not contain boundary
bridges are bounded by at most one element from $\frakB$ which is not affected in 
the construction of $\tilde P$, so are bounded by at most one element from $\tilde\frakB$.  As a 
result we have the following lemma.

\begin{lemma}
The $(m+2)$-angulation $(\tilde P,\tilde M, \tilde T)$ is  non-degenerate.
\end{lemma}

\begin{lemma}\label{lem degenerate Q}
The quivers with relations $(Q_T,I_T)$ and $(Q_{\tilde T},I_{\tilde T})$ are equal. 
\end{lemma}
\begin{proof}
This follows immediately from Remark~\ref{remark T}. The construction of the boundary 
bridges does not change the set $T$, so there is a clear bijection between $T$ and $\tilde T$.  
It follows that the set of vertices in $Q_T$ and $Q_{\tilde T}$ are the same. Similarly, the 
incidence of arcs in $T$ is not impacted by the construction of the bridges, so the
set of arrows and the set of relations is also the same. 
\end{proof}

\begin{thm}
Let $(P,M,T)$ be any $(m+2)$-angulation and $\alg$ the corresponding algebra. 
The AG-invariant of $\alg$ is given by
\[t(0,m+2)^* + \sum_{i}(a_i,b_i)^* \] 
where $t$ is number of internal $(m+2)$-gons in $(P,M,T)$, 
$i$ indexes the boundary components of $\tilde P$, and 
\[ a_i = \# \tilde M_T^i, \quad \text{and}\quad b_i= \sum_{\frakb_j\in\tilde\frakB_i}(m-w(\frakb_j)).\]
\end{thm}
\begin{proof}
If $(P,M,T)$ is non-degenerate, then this immediately reduces to Theorem~\ref{main thm 2}.  On
the other hand, when $(P,M,T)$ is degenerate, then  from Lemma~\ref{lem degenerate Q} we see
that $\alg$ also comes from $(\tilde P,\tilde M,\tilde T)$, which is non-degenerate. We can 
then apply Theorem~\ref{main thm 2} to $(\tilde P,\tilde M,\tilde T)$ to get the desired formula. 
\end{proof}

\printbibliography

\end{document}